\documentclass[a4paper,12pt]{amsart}
\usepackage{amssymb}
\usepackage[matrix,arrow,curve]{xy}

\author{Olga V. Chuvashova}
\address{
Department of Higher Algebra\\
Faculty of Mechanics and Mathematics\\
Moscow State University\\
119992 Moscow, Russia}
\email{chuvashova@gmail.com}

\author{Nikolay A. Pechenkin}
\address{
Algebraic Geometry Section\\
Steklov Mathematical Institute\\
Russian Academy of Sciences\\
Gubkin str. 8, GSP-1\\
119991 Moscow, Russia}
\email{kolia.pechnik@gmail.com}

\thanks{The second author was partially supported by the grants NSh-5139.2012.1, MK-6612.2012.1, RFFI 12-01-31506, RFFI 12-01-33024.}

\title{Quotients of an affine variety by~an~action~of~a~torus}



\newcommand{\Hom}{{\rm Hom}}

\newcommand{\Sp}{{\rm Spec\ }}
\newcommand{\Proj}{{\rm Proj\ }}

\newcommand{\G}{k^{\times}}
\newcommand{\Oo}{\mathcal{O}}
\newcommand{\C}{\mathcal{C}}

\renewcommand{\H}{\mathcal{H}}
\newcommand{\cQ}{\mathcal{Q}}

\newcommand{\Q}{\mathbb{Q}}
\newcommand{\A}{\mathbb{A}}
\newcommand{\X}{X}

\newcommand{\T}{\mathbb{T}}
\renewcommand{\P}{\mathbb{P}}
\newcommand{\Z}{\mathbb{Z}}
\newcommand{\N}{\mathbb{N}}
\newcommand{\cone}{{\rm cone}}
\newcommand{\conv}{{\rm conv}}
\newcommand{\norm}{{\rm norm}}
\newcommand{\ri}{{\rm relint}}
\newcommand{\supp}{{\rm supp}}

\newcommand{\Chi}{\mathfrak{X}}

\newtheorem{theorem}{Theorem}[section]
\newtheorem{corollary}[theorem]{Corollary}
\newtheorem{proposition}[theorem]{Proposition}
\newtheorem{lemma}[theorem]{Lemma}
\theoremstyle{definition}
\newtheorem{definition}[theorem]{Definition}
\newtheorem{example}[theorem]{Example}

\theoremstyle{remark}
\newtheorem{remark}[theorem]{Remark}

\begin{document}

\begin{abstract}
Let $X$ be an affine $T$-variety. We study two different quotients for the action of $T$ on $X$: the toric Chow quotient $X/_CT$ and the toric Hilbert scheme $H$. We introduce a notion of the main component $H_0$ of $H$, which parameterizes general $T$-orbit closures in $X$ and their flat limits. The main component $U_0$ of the universal family $U$ over $H$ is a preimage of $H_0$. We define an analogue of a universal family $W_X$ over the main component of the $X/_CT$. We show that the toric Chow morphism restricted on the main components lifts to a birational projective morphism from $U_0$ to $W_X$. The variety $W_X$ also provides a geometric realization of the Altmann-Hausen family. In particular, the notion of $W_X$ allows us to provide an explicit description of the fan of the Altmann-Hausen family in the toric case.
\end{abstract}

\maketitle

\section{Introduction}\label{intro}
An important problem in algebraic geometry is to introduce a good
notion of a quotient for an action of a reductive algebraic group $G$
on a variety $\X$. For many actions there exists an open subset
$U\subset \X$ where $G$ acts freely, such that a variety $U/G$
exists as a geometric quotient. Constructing a quotient $\X/G$ is,
thus, choosing a compactification of~$U/G$.

In the case when $\X$ is projective, one approach to this problem
is provided by geometric invariant theory (GIT) developed by
Mumford \cite{MF}. Given a $G$-equivariant embedding of $\X$ in
the projectivisation of a $G$-module, the GIT-quotient is the
projective spectrum of the subring of $G$-invariants in the
homogeneous coordinate ring on $\X$. There are two other natural compactifications, provided by
appropriate Chow varieties of algebraic cycles and Hilbert
schemes. The Chow quotient of a projective variety $\X$
parameterizes the closures of $G$-orbits in $\X$ having the same
dimension and degree and their limits in the Chow variety of all
algebraic cycles having these parameters. The Chow quotient of a
toric projective variety by a subtorus action was studied
in \cite{KSZ}. The invariant Hilbert scheme classifies closed $G$-invariant subschemes
$Z\subset X$ such that the $G$-module $\Oo(Z)$ has prescribed multiplicities (see \cite{B}). The $G$-Hilbert scheme is a particular case of the invariant Hilbert scheme. It arises in the case of a finite group $G$ and is considered in \cite{Be}. The main component of the $G$-Hilbert scheme parameterizes regular $G$-orbits in $X$ and their flat limits. Another particular case of an invariant Hilbert scheme is the toric Hilbert scheme which will be considered below.

\bigskip

We are interested in the following case. Let $X$ be an affine variety and $G=T$ be an algebraic torus. Denote by $X^s$ the subset of stable points in $X$ under the torus action which is an open subset in $X$ where $T$ acts freely (see Section \ref{mc} for the precise definition). The toric Hilbert scheme $H$ is defined as the invariant Hilbert scheme parameterizing $T$-invariant ideals in $k[X]$ having the same Hilbert function as the toric variety $\overline{Tx}$, where $x\in X^s$. The main component $H_0$ of $H$ is the irreducible component obtained by closing the image of $X^s/T$ in $H$. In the case of the subtorus action on the affine toric variety $X$ the toric Hilbert scheme and its main component were studied in \cite{C}.

The Chow quotient $(X/_C T)_0$ of $X$ by the $T$-action is by definition the main component of the inverse limit $X/_C T$ of $GIT$-quotients (see Section~\ref{git1}). It obtained by closing the image of $X^s/T$ in $X/_C T$, so there is a canonical rational map $q: X\dashrightarrow (\X/_C T)_0$. Denote by $Y:=(X/_C T)_0^{\text{norm}}$ the normalization of $(X/_C T)_0$. In \cite{AH} Altmann and Hausen introduced a certain family $\psi : \widetilde{X} \rightarrow Y$ of $T$-varieties, where the $T$-variety $\widetilde{X}$ can be interpreted as a resolution of $X$ improving the quotient behavior. The morphism $\psi$ is a good quotient under the action of $T$ and there is a proper and birational morphism $\varphi: \widetilde{X} \rightarrow X$  such that the following diagram is commutative
$$
\xymatrix{
\widetilde{X} \ar[r]^{\varphi} \ar[d]_{\psi} & X \ar@{-->}[ld]^{q^{\text{norm}}}\\
 Y.
}
$$
The family $\psi$ is closely related to the combinatorial-geometrical datum used in~\cite{AH} for the description of affine normal $T$-varieties. In Section~\ref{cah} we recall the construction of $\psi$ and explain how one can easily pass from $\psi$ to this datum and back.
A first result of the paper, Proposition~\ref{main},
realizes $\psi$ as the normalization of an elementarily constructed family
$p_C: W_X \rightarrow (\X/_CT)_0$. The variety $W_X$ is defined as the closure of the graph of the rational map $q$, and $p_C$ is the projection on~$(\X/_C T)_0$. In Theorem~\ref{veer} we then use this result to construct the fan of variety $\widetilde{X}$ in the toric case.

Our main result, Theorem~\ref{remw}, relates the universal
family $U_0 \rightarrow H_0$ over the main component of the toric Hilbert
scheme via a toric Chow morphism to the family $W_X \rightarrow (\X/_C T)_0$. The toric Chow morphism from $H$ to $X/_C T$ was constructed by Haiman and Sturmfels in \cite{HS} in the case when $X$ is a finite-dimensional $T$-module. We generalize their construction and include this morphism into a commutative square
\begin{equation}\label{cd}
\xymatrix{
         U_0 \ar[r]\ar[d]   & W_\X \ar[d]^{p_C} \\
         H_0 \ar[r]  & (\X/_C T)_0.
}
\end{equation}
We state that horizontal morphisms here are projective and birational.

\bigskip

This paper is organized as follows. In Section~\ref{sec2} we fix notations concerning toric geometry and $T$-varieties and recall the definition of the functor of points which is necessary to define the invariant Hilbert scheme. In Section~\ref{sec3-1} we provide an exposition of the results of \cite{AH}. We recall the notion of the Altmann-Hausen family and then prove our first result, Proposition~\ref{main}, about its simple geometrical realization. In Section~\ref{sec3-2} we recall the definition of the toric Hilbert scheme and introduce the notion of its main component in the $T$-variety case. In Section~\ref{sec3-3} we investigate the toric Chow morphism, lift it to the universal families and prove our main result, Theorem~\ref{remw}. Section~\ref{sec4} is, in fact, repetition of the Section~\ref{sec3} in the case of a subtorus action on an affine toric variety. In this case almost all the statements of Section~\ref{sec3} can be interpreted in terms of fans. We provide the construction of the fans of the toric varieties $H_0$, $(X/_C T)_0$, $U_0$ and $W_X$ and compute them in concrete examples.

\section{Preliminaries}\label{sec2}

We consider the category of schemes over an algebraically closed
field $k$ of characteristic zero. Our main references on schemes are \cite{EH} and \cite{Har}.

A {\it variety} is a separated reduced scheme of finite
type. We denote by $\Oo_Z$ the structure sheaf of a scheme $Z$, and
$k[Z]$ denotes the algebra of sections of $\Oo_Z$ over $Z$.

\subsection{Basic facts from toric geometry}

An $n$-dimensional {\it torus} $T$ is an algebraic group isomorphic to the direct product of $n$ copies of the multiplicative group $\G$.
We use the notations $\Chi(T) = \Hom(T, \G)$ and $\Lambda(T)= \Hom(\G, T)$ for the lattices of characters and one-parameter subgroups of $T$ respectively. Denote by $\Chi(T)_{\Q}:=\Chi(T)\otimes_{\Z}\Q$ and $\Lambda(T)_{\Q}:=\Lambda(T)\otimes_{\Z}\Q$ the corresponding $\Q$-vector spaces.

For any affine scheme $X$ with an action of a
torus $T$  its algebra of regular functions $k[\X]$ is graded by the group $\Chi(T)$ of characters of
$T$ : $$k[X]=\bigoplus_{\chi\in\, \Chi(T)} k[X]_{\chi},$$
 where  $k[X]_{\chi}:=\{f\in k[X]: t\cdot f=\chi(t)f \,\,\,\, \forall \, t\in T \}$ is the
subspace of $T$-semi-invariant functions of weight $\chi$. Let
$$\Sigma:=\{\chi\in \Chi(T)\ : \ k[X]_\chi\ne 0\}.$$
If $X$ is an irreducible variety, then $\Sigma$ is a finitely generated monoid called the \emph{weight monoid}. If $T$ acts on $X$ faithfully, then $\Sigma$ generates $\Chi(T)$.

A $T$-\emph{variety} is a normal variety endowed with a faithful regular action of $T$. Let $X$ be a $T$-variety. A morphism $\pi: X \to Y$ is called a \emph{good quotient} for this action if $\pi$ is affine, $T$-invariant, and the canonical map $\pi^{\sharp}: \mathcal{O}_Y \rightarrow \pi_{\ast}(\mathcal{O}_X)^T$ is an isomorphism.

Given a scheme $S$, a {\it family of affine $T$-schemes} over $S$
is a scheme $X$ equipped with an action of $T$ and with a morphism
$p:X\to S$ such that $p$ is affine, of finite type and
$T$-invariant. Then  the sheaf of $\Oo_S$-algebras $p_*(\Oo_X)$ is
equipped with a compatible grading by $\Chi(T)$.

A {\it toric variety} under the torus $T$ is an irreducible $T$-variety $X$ that contains an open orbit isomorphic to $T$. We do not require $X$ to be normal. We will consider only those toric varieties $X$ that admit an open covering by affine $T$-invariant charts (all normal toric varieties satisfy this condition). Our main references on toric varieties are \cite{CLS}, \cite{Ful} and \cite{Oda}.

 Given a toric variety $X$, we denote by $\C_X\subset \Lambda(T)_{\Q}$ the associated fan. The $T$-orbits on $X$ are in order-reversing one-to-one correspondence with the cones of $\C_X$. If $\sigma(Z)$ is the cone in $\C_X$ corresponding to a $T$-orbit $Z$, then a one-parameter subgroup $\lambda\in \Lambda(T)$ lies in the interior of $\sigma(Z)$ if and only if $\lim_{s\to 0}\lambda (s)$ exists and lies in $Z$. A toric variety is determined by its fan up to normalization.

Let $X_1$ under a torus $T_1$ and $X_2$ under a torus~$T_2$ be toric varieties. A morphism $\varphi: X_1\rightarrow X_2$ is \emph{toric} if $\varphi$ maps the torus $T_1$ into $T_2$ and $\varphi |_{T_1}: T_1\rightarrow T_2$ is a group homomorphism. For the toric morphism $\varphi$ we have the following commutative diagram:
$$
\xymatrix{
T_1\times X_1 \ar[d]^{\varphi |_{T_1}\times\varphi}\ar[r] & X_1 \ar[d]^{\varphi} \\
T_2\times X_2 \ar[r] & X_2 .
}
$$

 The morphism of algebraic groups $\varphi: T_1\rightarrow T_2$ induces a $\Z$-linear map $\overline{\varphi}: \Lambda(T_1)\rightarrow \Lambda(T_2)$ and $\Q$-linear map $\overline{\varphi}_{\Q}: \Lambda(T_1)_{\Q}\rightarrow \Lambda(T_2)_{\Q}$.

In the following two propositions $X_1$ and $X_2$ are normal toric varieties.
\begin{proposition}\cite[Theorem~3.3.4]{CLS}\label{tormorph}
 A morphism $\varphi: X_1\rightarrow X_2$ is toric if and only if the corresponding $\Z$-linear map $\overline{\varphi}$ is compatible with fans $\mathcal{C}_{X_1}$ and $\mathcal{C}_{X_2}$. It means that for every cone $\varsigma\in \mathcal{C}_{X_1}$ there exists a cone $\widetilde{\varsigma}\in \mathcal{C}_{X_2}$ such that $\overline{\varphi}_{\Q}(\varsigma)\subset \widetilde{\varsigma}$.
\end{proposition}
Denote by $|\mathcal{C}|$ the support of the fan $\mathcal{C}$.
\begin{proposition}\cite[Theorem~3.4.11]{CLS}\label{prop}
A toric morphism $\varphi: X_1\rightarrow X_2$ is proper if and only if $\overline{\varphi}_{\Q}^{-1}(|\mathcal{C}_{X_2}|)=|\mathcal{C}_{X_1}|$.
\end{proposition}

If $X$ is toric variety under a factor torus $\T/T$, then we can consider the fan of this variety as a quasifan in $\Lambda(\T)_{\Q}$ whose cones include the subspace $\Lambda(T)_{\Q}\subset \Lambda(\T)_{\Q}$.

\begin{proposition}\label{per}
Let $X=\overline{Tx}$, $Y=\overline{Ty}$ be toric varieties. Then the fan associated to the toric variety $\overline{T(x,y)}\subset X\times Y$ is the coarsest common refinement of fans $\mathcal{C}_X$ and $\mathcal{C}_Y$ (in particular, its support is equal to $|\mathcal{C}_X|\cap|\mathcal{C}_Y|$).
\end{proposition}
Note that the action of the torus $T$ on the toric varieties $X$ and $Y$ is not required to be faithful here. So the fans $\mathcal{C}_X$ and $\mathcal{C}_Y$ should be considered as quasifans in $\Lambda(T)$.
\begin{proof}
 Let $\lambda\in \Lambda(T)$. The limit $\lim_{t\rightarrow 0} \lambda (t)(x,y)$ exists if and only if both limits $\lim_{t\rightarrow 0}\lambda (t)x$ and $\lim_{t\rightarrow 0}\lambda (t)y$ exist. The limits of the point $(x,y)$ with respect to one-parameter subgroups $\lambda_1$ and $\lambda_2$ coincide if and only if the limits with respect to $\lambda_1$ and $\lambda_2$ of both points $x$ and $y$ coincide, i.e. $\lambda_1$ and $\lambda_2$ lie in the same cones of the fans $\mathcal{C}_X$ and $\mathcal{C}_Y$.
\end{proof}

\subsection{Basic facts on the functor of points}

Recall
that any scheme $Z$ is characterized by its {\it functor of
points} that is the contravariant functor from the category of schemes to the category of sets:
$$\underline{Z}: ({\rm Sch })^{\circ} \to ({\rm Set}), \ \ \
\underline{Z}(X):={\rm Mor}(X, Z),$$ where ${\rm Mor}(X, Z)$ is the set of morphisms of schemes
from $X$ to $Z$ over $k$ (we denote the functor of points of a scheme by the corresponding underlined letter). Each $f\in {\rm Mor}(X, Y)$ defines a morphism of sets $\underline{Z}(f): \underline{Z}(Y)\to \underline{Z}(X)$. For $g\in {\rm Mor}(Y, Z)$ we have $\underline{Z}(f)(g):=g \circ f \in {\rm Mor}(X, Z)$.

Let $F: ({\rm Sch })^{\circ}\to ({\rm Set})$ be an arbitrary functor. We say that a scheme $Z$ represents the functor $F$ if there exists an isomorphism of functors $\underline{Z}\cong F$ (the scheme $Z$ is then also called the fine moduli space of the functor $F$). Denote by ${\rm (Fun((Sch)}^{\circ}{\rm,\,(Set)))}$ the category of contravariant functors from the category of schemes to the category of sets.

The covariant functor
$$\underline{\ast}: {\rm (Sch) } \to {\rm (Fun((Sch)}^{\circ}{\rm,\,(Set)))} $$
is defined by $X\mapsto \underline{X}$. By Yoneda's lemma, the functor $\underline{\ast}$ is an equivalence between the category of schemes and the full subcategory of the category of functors. In particular, it defines a natural bijection between the sets ${\rm Mor}(X, Y)$ and ${\rm Mor}(\underline{X}, \underline{Y})$.

\section{General constructions}\label{sec3}

\subsection{The toric Chow quotient and the Altmann-Hausen theory}\label{sec3-1}
\subsubsection{Combinatorial description of affine normal T-varieties} In this section we recall the description of normal affine $T$-varieties in terms of proper polyhedral divisors on a normal semiprojective variety given in \cite{AH}.

 Let $M$ be a lattice, $N$ be its dual lattice, $M_{\Q}:=M\otimes_{\Z}\Q$, $N_{\Q}:=N\otimes_{\Z}\Q$. Denote by $\langle \cdot , \cdot \rangle : M_{\Q}\times N_{\Q} \rightarrow \Q$ the natural duality pairing, by $\sigma$ a pointed polyhedral cone in $N_{\Q}$, by $\sigma^{\vee}\subset M_{\Q}$ its dual cone, and by $\text{Pol}_{\sigma}(N_{\Q})$ the set of all $\sigma$-polyhedra, i.e. polyhedra in $N_{\Q}$ with the recession cone $\sigma$ (see \cite[Definition~1.11]{Z} for the definition of the recession cone). Let $T := \text{Spec} \ k[M]$ be an algebraic torus with its lattice of characters equal to $M$.

To a $\sigma$-polyhedron $\Delta\in \text{Pol}_{\sigma}(N_{\Q})$ we
associate its support function $h_{\Delta}: \sigma^{\vee} \rightarrow \Q$, defined by $$h_{\Delta}(m)=\min \langle m, \Delta \rangle=\min_{p\in\Delta} \langle m, p \rangle.$$

\begin{definition}
 A variety is called \emph{semiprojective} if it is projective over an affine variety. Let $Y$ be a normal semiprojective variety. A $\sigma$\emph{-polyhedral divisor} on $Y$ is a formal sum $\mathfrak{D}=\sum_Z \Delta_Z \cdot Z$, where $Z$ runs over all prime divisors on $Y$, $\Delta_Z\in \text{Pol}_{\sigma}(N_{\Q})$, and $\Delta_Z=\sigma$ for all but finitely many $Z$. For $m\in\sigma^{\vee}$ we define the $\Q$-divisor $\mathfrak{D}(m):=\sum_Z h_{\Delta_Z}(m)\cdot Z$ on $Y$. A $\sigma$-polyhedral divisor D is called \emph{proper} if the following two conditions hold:
\begin{enumerate}
\item
$\mathfrak{D}(m)$ is semiample and $\Q$-Cartier for all $m\in\sigma^{\vee}$,
\item
$\mathfrak{D}(m)$ is big for all $m\in\ri(\sigma^{\vee})$.
\end{enumerate}
\end{definition}

 A $\Q$-Cartier divisor $D$ is called \emph{semiample} if there exists $r > 0$, such that the linear system $|rD|$ is base point free, and \emph{big} if there exists a divisor $D_0\in|rD|$, for some $r > 0$, such that the complement $Y\backslash \text{Supp}(D_0)$ is affine.

For a $\Q$-divisor $D=\sum_Z a_Z\cdot Z$, let $\lfloor D\rfloor := \sum_Z \lfloor a_Z \rfloor \cdot Z$ be the round-down divisor of~$D$, and $\mathcal{O}(D):=\mathcal{O}(\lfloor D\rfloor)$ be the corresponding sheaf of $\mathcal{O}_Y$-modules. Any $\sigma$-polyhedral divisor $\mathfrak{D}$ defines an $M$-graded quasicoherent sheaf of algebras on~$Y$ as follows:
$$\mathcal{A}[Y,\mathfrak{D}]:=\bigoplus_{m\in \sigma^{\vee}\cap M} \mathcal{O}(\mathfrak{D}(m)).$$
Let $A[Y,\mathfrak{D}]:=\Gamma (Y, \mathcal{A}[Y,\mathfrak{D}])$ be the $M$-graded algebra corresponding to $\mathfrak{D}$.

\begin{theorem}\label{a-h}
To any proper $\sigma$-polyhedral divisor $\mathfrak{D}$ on a normal semiprojective variety~$Y$
one can associate a normal affine $T$-variety of dimension $\dim Y+\dim T$ defined by $X[Y, \mathfrak{D}]:=\Sp A[Y,\mathfrak{D}]$.

Conversely, any normal affine $T$-variety is isomorphic to $X[Y,\mathfrak{D}]$ for some semiprojective
variety $Y$ and some proper $\sigma$-polyhedral divisor $\mathfrak{D}$ on it.
\end{theorem}

We will discuss the converse correspondence in Section \ref{cah}.

\subsubsection{Ingredients from GIT}\label{git1}

Let $X$ be an affine $T$-variety. Let us recall the construction of the inverse limit of GIT-quotients.

  The cone $\omega$ generated by the weight monoid $\Sigma$ in the vector space $\mathfrak{X}(T)_{\Q}:=\mathfrak{X}(T)\otimes_{\Z}\Q$ is called the \emph{weight cone}. For any $\chi\in \Sigma$, consider the set of its semistable points
$$X^{ss}_{\chi}:=\bigcup_{r>0}\bigcup_{f\in k[X]_{r\chi}}X_f.$$
 Two characters $\chi_1 \in \Sigma$ and $\chi_2 \in \Sigma$ are called \emph{equivalent} if $X^{ss}_{\chi_1}=X^{ss}_{\chi_2}$. Under this equivalence $\Sigma$ decomposes into finitely many equivalence classes which are forming the GIT-fan $\cQ \subset \mathfrak{X}(T)_{\Q}$ with $\supp (\cQ)= \omega$ (see \cite[Section 2]{BH}). Recall some definitions giving us the construction of the GIT-fan.

\begin{definition} \label{orb} \cite[Definition~2.1]{BH} For any $x\in
X$, the {\it orbit cone} $w_x$ associated to $x$ is the following
convex cone in $\Chi(T)$:
$$w_x:=\cone\{\chi\in \Sigma : \exists\ f\in k[X]_{\chi} {\rm\ such \ that \ } f(x)\ne
0\}.$$
\end{definition}
\begin{definition} \label{git} \cite[Definition~2.8]{BH}
For any character $\chi\in \Sigma$, the associated {\it
GIT-cone}~$\sigma_\chi$ is the intersection of all orbit cones
containing $\chi$: $$\sigma_\chi=\bigcap_{\{x\in X : \chi\in
w_x\}}w_x.$$
\end{definition}

By \cite[Theorem~2.11]{BH}, the collection of GIT-cones forms a fan
$\cQ$ having $\omega$ as its support. Moreover, the following statement holds.
\begin{proposition} \cite[Proposition~2.9]{BH} Let $\chi_1, \chi_2 \in \Sigma$. Then
$X^{ss}_{\chi_1}\subseteq X^{ss}_{\chi_2}$ if and only if
$\sigma_{\chi_1}\supseteq \sigma_{\chi_2}$.
\end{proposition}
 For any cone $\lambda\in \cQ$, denote $X^{ss}_{\lambda}:=X^{ss}_{\chi}$, where $\chi$ is an arbitrary character in $\ri(\lambda)$. Let $X/_{\!\!\lambda} T := X^{ss}_{\lambda}\!/\!/T$ be the good quotient under the action of $T$. Varieties $X/_{\!\!\lambda} T$ are called \emph{GIT-quotients}. Notice also that $$X/_{\!\!\lambda} T=\Proj k[X]^{(\chi)}$$ for any $\chi\in \ri(\lambda)$, where $$k[\X]^{(\chi)}:=\bigoplus_{r=0}^{\infty} k[\X]_{r\chi}.$$ In particular, $X/_{\! 0}T=X/\!/T=\Sp k[X]^T$.

 Denote by $q_{\lambda}: X^{ss}_{\lambda}\to X/_{\!\!\lambda}T$ the quotient map.
  We consider natural morphisms between GIT-quotients. Namely, if $\lambda_1 \supset \lambda_2$, where
$\lambda_1, \lambda_2\in \cQ$, then we have the following commutative
diagram:
$$
\xymatrix{
X^{ss}_{\lambda_1} \ar@{->}[d]^{q_{\lambda_1}}\ar@{^{(}->}[r] & X^{ss}_{\lambda_2} \ar@{->}[d]^{q_{\lambda_2}} \\
X/_{\lambda_1}T \ar[r]^{p_{\lambda_1 \lambda_2}} & X/_{\lambda_2}T
}
$$
So the quotient maps $q_{\lambda}: X^{ss}_{\lambda}\to X/_{\!\!\lambda} T$
form a finite inverse system with $q_0: X\to X/\!/T$ sitting at
the end. We have a canonical morphism
$$q : X^{ss}:=\bigcap_{\chi\in
\Sigma}X^{ss}_{\chi}\to{\lim_{\longleftarrow}}\ X/_{\!\!\chi}
T=:X/_CT.$$
The variety $X/_CT$ is called \emph{GIT-limit}.

\begin{definition}
The {\it main component} $(X/_CT)_0$ of the GIT-limit $X/_CT$ is the closure of the image $q(X^{ss})\subset X/_CT$.
\end{definition}
The main component of the GIT-limit is also called the \emph{toric Chow quotient} (see e.g. \cite{CM}). This terminology corresponds to results of \cite{KSZ}, where it was proved that in the case of a projective toric variety $X$ the main component of GIT-limit is, indeed, isomorphic to the Chow quotient by the action of a subtorus.

\subsubsection{The Altmann-Hausen family}\label{cah}

Now we give a construction of the Altmann-Hausen family of an affine normal $T$-variety $X$ following \cite[Section~6]{AH}.

   Let $Y:=(X/_C T)_0^{\norm}$ be the normalization of the main component of the GIT-limit, $q^{\norm}: X^{ss}\rightarrow Y$ be the lifting of $q$ to the normalization, and $p_{\lambda}: Y\rightarrow X/_{\lambda} T$ be the composition of the normalization map and canonical morphisms from GIT-limit to the elements of the inverse system restricted on the main component. For $\lambda_1,\lambda_2\in \cQ,\,\lambda_1\supset \lambda_2$ we have the following commutative diagram:
$$
\xymatrix{
X^{ss} \ar[d]^{q^{\norm}} \ar@{^{(}->}[r] & X^{ss}_{\lambda_1} \ar[d]^{q_{\lambda_1}}\ar@{^{(}->}[r] & X^{ss}_{\lambda_2} \ar[d]^{q_{\lambda_2}} \ar@{^{(}->}[r] & X \ar[ddd]^{q_0}\\
Y \ar[r]^{p_{\lambda_1}} \ar[rrrdd]^{p_0} & X/_{\lambda_1}T \ar[r]^{p_{\lambda_1 \lambda_2}} \ar[rrdd]^{p_{\lambda_1 0}} & X/_{\lambda_2}T \ar[rdd]^{p_{\lambda_2 0}}\\
\\
& & & X/\!/T
}
$$
 As it was shown in \cite[Lemma~6.1]{AH}, the morphisms $p_{\lambda}$ and $p_{\lambda_1 \lambda_2}$ are projective surjections with connected fibers. In particular, $Y$ is projective over the affine variety $X/\!/T=\Sp k[X]^T$. Moreover, the morphisms $p_{\lambda}$ are birational when $\lambda\cap \ri(\sigma)\ne \emptyset$, and, consequently, $\dim Y = \dim X - \dim T$ and $k(Y)=k(X)^T$.

\bigskip

 A character $\chi\in \Sigma$ is called \emph{saturated} if the algebra $k[X]^{(\chi)}$ is generated by elements of degree one. It is well known that for any $\chi\in \Sigma$ there exists an integral $n_{\chi}>0$ such that  $kn_{\chi}\chi$ is saturated for all integral $k>0$. For multigraded version of this notion and an algebraic characterization of GIT-fan see \cite{ArH}.

 For any $\lambda\in \cQ$ and $\chi \in \text{relint}(\lambda)$, we have a sheaf of $\mathcal{O}_{X/_{\lambda}T}$-modules $\mathcal{A}_{\lambda,\chi}:=(q_{\lambda})_{\ast}(\mathcal{O}_{X^{ss}_{\lambda}})_{\chi}$. It is easy to see that the defined sheaf coincides with the twisting sheaf of Serre, if the variety $X/_{\lambda}T$ is considered as $\Proj k[X]^{(\chi)}$. In particular, the sheaf $\mathcal{A}_{\lambda,\chi}$ is invertible if $\chi$ is saturated. For a saturated $\chi$, denote by $\mathcal{A}_{\chi}:=(p_{\lambda(\chi)})^{\ast}(\mathcal{A}_{\lambda(\chi),\chi})$ the invertible sheaves of $\mathcal{O}_Y$-modules.

Firstly, assume that $\chi$ is saturated. If $f\in k[X]_{\chi}$, $\lambda:=\sigma_{\chi}\in \cQ$, then $X_f/\!/T$ is an open affine subset in $X/_{\lambda} T$ and $\mathcal{A}_{\lambda,\chi}|_{X_f/\!/T}=f\cdot \mathcal{O}_{X_f/\!/T}$ (see \cite[Lemma~6.3(ii)]{AH}). Denote $Y_f:=p_{\lambda}^{-1}(X_f/\!/T)$. The open sets $Y_f$ cover $Y$, and by definition of the inverse image we have (see \cite[Lemma~6.4(ii)]{AH})
$$ \mathcal{A}_{\chi}|_{Y_f}=f\cdot \mathcal{O}_{Y_f}\subset f\cdot k(Y)=f\cdot k(X)^T=k(X)_{\chi}.$$
 Consequently, we can consider the sheaves $\mathcal{A}_{\chi}$ as subsheaves of the constant sheaf $k(X)$ on~$Y$.

   For unsaturated $\chi$, we define
$$\mathcal{A}_{\chi}(U):=\{ f\in k(X): f^{n_{\chi}}\in \mathcal{A}_{n_{\chi}\chi}(U) \}.$$
 The defined sheaves together constitute the $\Chi(T)$-graded sheaf of $\mathcal{O}_Y$-algebras $$\mathcal{A}:=\bigoplus_{\chi\in \Sigma} \mathcal{A}_{\chi}.$$ From \cite[Lemma~6.4]{AH} we can see that the multiplication is defined correctly and preserves the grading.
 \begin{definition}The \emph{Altmann-Hausen family} is the family $\psi: \widetilde{X}\rightarrow Y$ of $T$-schemes, where
$\widetilde{X}:=\text{Spec}_Y \ \mathcal{A}$ is the relative spectrum, $\psi$ is the morphism given by $\mathcal{A}_0\rightarrow \mathcal{A}$ and the $T$-action on $\widetilde{X}$ is given by the $\mathfrak{X}(T)$-grading of $\mathcal{A}$.
\end{definition}
By \cite[Lemma~6.4(ii)]{AH}, we have $X=\Sp \Gamma (\widetilde{X},\mathcal{O}_{\widetilde{X}})$.

\bigskip

Summarizing all that we have done in this section, note that from a $T$-variety $X$ we constructed a semiprojective variety $Y$ equipped with a $\Chi(T)$-graded sheaf of $\mathcal{O}_Y$-algebras~$\mathcal{A}$ such that $X$ can be restored from this data.

Denote $\Lambda(T)_{\Q}:=\Lambda(T)\otimes_{\Z} \Q$ and $\sigma:=\omega^{\vee}\subset \Lambda(T)_{\Q}$. One can easily pass from the sheaf of algebras $\mathcal{A}$ to the proper $\sigma$-polyhedral divisor $\mathfrak{D}$ on $Y$ such that
\begin{equation}\label{e3}
\mathcal{A}=\mathcal{A}[Y,\mathfrak{D}].
 \end{equation}
 To do this, let's choose a (non-canonical) homomorphism $s: \mathfrak{X}(T)\rightarrow k(X)$ such that for every $\chi\in\mathfrak{X}(T)$ the function $s(\chi)$ is homogeneous of degree $\chi$. Such homeomorphisms always exist since $T$ acts on $X$ faithfully. Next, for a saturated $\chi\in \omega\cap \mathfrak{X}(T)$ there exists an unique Cartier divisor $\mathfrak{D}(\chi)$ on $Y$ such that $$\mathcal{O}(\mathfrak{D}(\chi))=\frac{1}{s(\chi)}\cdot \mathcal{A}_{\chi}\subset k(Y).$$ For unsaturated $\chi$, define $\mathfrak{D}(\chi):=\frac{1}{n_{\chi}}\cdot \mathfrak{D}(n_{\chi}\chi)$. Now one can check that $\Q$-Cartier divisors $\mathfrak{D}(\chi)$ can be ``glued'' together to form a proper $\sigma$-polyhedral divisor $\mathfrak{D}$ on $Y$ satisfying condition~(\ref{e3}) (see \cite[Section~6]{AH}). The divisor $\mathfrak{D}$ constructed in such way is called \emph{minimal}.

We see that the variety $\widetilde{X}$ appears as a middle step in the correspondence between normal affine $T$-varieties and combinatorial data of this action  $(Y,\mathfrak{D})$ from Theorem \ref{a-h} with an additional minimality condition on $\mathfrak{D}$:
$$
\xymatrix{
& T\curvearrowright\widetilde{X}\ar@/_-1pc/@{-->}[dr]\ar@/_1pc/@<-1ex>[dl] \\
 T\curvearrowright X \ar@/_-1pc/@{->}[ur]  & & \{(Y,\mathfrak{D})\}, \ar@/_1pc/@<-1ex>[ul]
}
$$
where one of the arrows is dashed, because the reconstruction of $\mathfrak{D}$ from $\widetilde{X}$ is not canonical.

 The main result illustrating the role of the Altmann-Hausen family in this assignment is the following one.
\begin{proposition}\cite[Theorem~3.1]{AH}\label{ahm}
The morphism $\psi$ is a good quotient for the $T$-action. The canonical morphism $\varphi: \widetilde{X}\rightarrow X$ is $T$-equivariant, birational and proper. We have the following commutative diagram:
$$
\xymatrix{
X^{ss} \ar@{^{(}->}[r] \ar[dr]_-{q^{\norm}} & X \ar@<-0.5ex>@{-->}[r]_{\varphi^{-1}} \ar@{-->}[d] & \widetilde{X} \ar@<-0.5ex>[l]_{\varphi} \ar[ld]^{\psi} \\
& Y
}
$$
\end{proposition}

\subsubsection{A geometric description of the Altmann-Hausen family}
Define $$W_X:=\overline{\{(x,q(x)):x\in X^{ss}\}}\subset X \times (X/_C T)_{0},$$ and let $p_C$ be projection on the second component.

The following result modulo normalization morphism was stated in \cite[Lemma~1]{V}, where actions of complexity one were considered, but the same proof is valid for a general case.
\begin{proposition}\label{main}
 The Altmann-Hausen family $\psi : \widetilde{X} \rightarrow Y$ is isomorphic to the normalization $W_X^{\text{norm}} \rightarrow (X/_C T)_{0}^{\text{norm}}$ of the family of $T$-schemes $p_C : W_X\rightarrow (X/_C T)_{0}$.
\end{proposition}
\begin{proof}
Define a morphism $\theta: \widetilde{X} \rightarrow (X/_C T)_{0}$ as a composition $\iota\circ\psi$ of the morphism $\psi : \widetilde{X} \rightarrow (X/_C T)_{0}^{\text{norm}}$ and the normalization morphism $\iota: (X/_C T)_{0}^{\text{norm}}\rightarrow (X/_C T)_{0}$. Due to commutativity of the diagram from Proposition \ref{ahm}, the mappings $\varphi: \widetilde{X}\rightarrow X$ and $\theta: \widetilde{X} \rightarrow (X/_C T)_{0}$ define a morphism $\alpha:=(\varphi,\theta): \widetilde{X} \rightarrow W_X$.

$$
\xymatrix{
\widetilde{X} \ar[dr]^{\alpha}\ar@/_-1pc/[ddrrr]^{\varphi}\ar@/_1pc/[dddrr]_{\theta} \\
& W_X \ar@{^{(}->}[dr]\\
&& X\times (X/_CT)_0 \ar[r]\ar[d] & X\\
&& (X/_CT)_0
}
$$

The morphism $\iota$ is affine as a normalization morphism, $\psi$ is affine as a good quotient, so $\theta$ is affine as a composition of affine morphisms. The morphism $\varphi$ is proper, so $\alpha$ is proper. Further, $\alpha$ is finite, because it is both proper and affine. But $\alpha$ is also birational, hence it is the normalization morphism.
\end{proof}

Note that the variety $W_X$ is equipped with the natural faithful action of the torus $T$, which lifts to an action of $T$ on $\widetilde{X}=W_X^{\text{norm}}$ and defines a $\mathfrak{X}(T)$-grading of the sheaf $\psi_{\ast}(\mathcal{O}_{\widetilde{X}})$. By formula $\mathcal{A}_{\chi}=\psi_{\ast}(\mathcal{O}_{\widetilde{X}})_{\chi}$, one can recover the $\mathfrak{X}(T)$-graded sheaf of $\mathcal{O}_Y$-algebras $\mathcal{A}$, and, as well, the proper polyhedral divisor on $Y$ (see Section \ref{cah}).

Denote by $p_C$ the projection $W_\X\to (\X/_CT)_0$. Note that $p_C$ is the good
quotient. Indeed, it is clear that $p_C$ is affine, $T$-invariant, and $p_C^{\sharp}: \mathcal{O}_{(\X/_CT)_0} \rightarrow (p_C)_{\ast}(\mathcal{O}_{W_X})^T$ is injective. To prove the surjectiveness of $p_C^{\sharp}$, use the fact that functions which are constant on the fibers come from the base.

We have the following commutative diagram:

$$
\xymatrix{
W_\X \ar[d]_{p_C} \ar[r]^{p^W_{\X}} & \X \ar[d]^{\pi} \\
(\X/_CT)_0  \ar[r]_{p^W_{\X}/\!/T} & \X/\!/T,
}
$$
where $p^W_{\X}$ denotes the projection on $\X$.

\medskip

\begin{example}Let $X=\A^n$, $T=k^{\times}$ and $T$ act on $X$ by scalar multiplication, i.e. $t\cdot (x_1,\ldots,x_n)=(tx_1,\ldots,tx_n)$. Then the $\Z$-grading of the algebra $k[X]=k[x_1,\ldots,x_n]$ is given by $\deg(x_i)=1$ for $i=1,\ldots,n$.

For this action, GIT-fan consists of two cones $\lambda_0$ and $\lambda_1$:
\begin{center}
\begin{picture}(90,20)
\put(0,10){\line(1,0){90}}
\put(45,10){\circle*{3}}
\linethickness{0.34mm}
\put(45,10){\line(1,0){45}}
\put(42,0){$0$}
\put(63,0){$1$}
\put(17,0){$-1$}
\put(42,13){$\lambda_0$}
\put(68,13){$\lambda_1$}
\end{picture}
\end{center}
 We have $$X^{ss}_{\lambda_0}=X,$$ $$X/_{\lambda_0} T=X/\!/T=\Sp k[X]^T=\{pt\};$$  $$X^{ss}_{\lambda_1}=X\backslash\{x_1=\ldots=x_n=0\},$$ $$X/_{\lambda_1}T=\Proj k[x_1,\ldots,x_n]=\P^{n-1}.$$ The inverse limit $X/_C T$ of GIT-system is equal to the variety $X/_{\lambda_1}T$, so $$Y=(X/_C T)_0=X/_C T=X/_{\lambda_1}T=\P^{n-1},$$ $$X^{ss}=X^{ss}_{\lambda_1}=X\backslash\{x_1=\ldots=x_n\}.$$ In coordinates the map $q^{\norm}: X^{ss}\rightarrow Y$ is given by $q^{\norm}(x_1,\ldots,x_n)=[x_1:\ldots :x_n]$. The variety $W_X\subset X\times (X/_C T)_0 =\A^n\times \P^{n-1}$ is defined by the equations $x_iy_j=x_jy_i$, where $(x_1,\ldots,x_n)$ are coordinates on $\A^n$, and $[y_1:\ldots :y_n]$ are homogeneous coordinates on $\P^{n-1}$, i.e. $W_X$ is the blowing-up of $\A^n$ at the point $(0,\ldots,0)$. On the other hand  $$\widetilde{X}=\text{Spec}_Y \mathcal{A}=\text{Spec}_Y \bigoplus_{n\in \Z_{\geqslant 0}} \mathcal{A}_n=$$ $$=\text{Spec}_Y\, \mathcal{O}_Y\oplus \bigoplus_{n\in \Z_{> 0}} \mathcal{A}_{\lambda_1,n}=\text{Spec}_Y \bigoplus_{n\in \Z_{\geqslant 0}} \mathcal{O}(n)=\text{Tot}(\mathcal{O}(-1))$$ is the total space of the tautological line bundle on $\P^{n-1}$. It is well-known that these varieties are isomorphic.
\end{example}

\subsection{The toric Hilbert scheme and its main component}\label{sec3-2}

Let $\X$ be an irreducible affine $T$-variety. In the following section we give an exposition of basic properties on multigraded Hilbert schemes and study some particular cases of this notion. See \cite[Section~3]{B} for some generalizations of these results on the case of $G$-variety $X$ with a reductive algebraic group $G$.

\subsubsection{Definitions and basic facts on multigraded Hilbert schemes}

\begin{definition}
 Given a function $h:\Chi(T)\to \N$, a family $p: F\to S$ of affine $T$-schemes \emph{has Hilbert function} $h$ if for every $\chi\in \Chi(T)$ the sheaf of $\Oo_S$-modules $p_*(\Oo_F)_\chi$ is locally free of rank $h(\chi)$.
\end{definition}

Note that the morphism $p$ is flat. If $h(0)=1$, then $p$ is the good quotient of $X$ by the action of $T$.

The following definition was introduced in \cite{HS}.

\begin{definition}\label{Def}  Given a function $h:\Chi(T)\to \N$, the {\it Hilbert functor}
is the contravariant functor $\H^h_{\X, T}$ from the category of schemes to the category of sets
assigning to any scheme $S$ the set of all closed $T$-stable subschemes $Z\subseteq S\times \X$ such that
the projection $p:Z\to S$ is a (flat) family of affine $T$-schemes with Hilbert function $h$.
\end{definition}

In \cite[Theorem 1.1]{HS} it was proved that there exists a quasiprojective scheme $H^h_{V, T}$ which represents
this functor in the case when $\X$ is a finite-dimensional $T$-module $V$; the scheme $H^h_{V, T}$ is called the {\it multigraded Hilbert scheme}.
In the case of an arbitrary $\X$ there exists a $T$-equivariant closed immersion $\X\hookrightarrow V$, where $V$ is a finite-dimensional $T$-module.
Then the Hilbert functor $\H^h_{\X, T}$ is represented by a closed subscheme $H^h_{\X, T}$ of $H^h_{V, T}$ (see
\cite[Lemma 1.6]{AB}). The scheme $H^h_{\X, T}$ is called the {\it invariant Hilbert scheme}.

Recall that the {\it universal family} $U^h_{\X, T}$ is the element of $\H^h_{\X, T}(H^h_{\X, T})$ corresponding to the identity map $\{ {\rm Id}: H^h_{\X, T} \to H^h_{\X, T}\}\in$ $ \underline{H^h_{\X, T}}(H^h_{\X, T})={\rm Mor}(H^h_{\X, T}, H^h_{\X, T})$. So $U^h_{\X, T}$ is the closed subscheme of $H^h_{\X, T}\times \X$ such that for any
$Z\in \H^h_{\X, T}(S)$ we have $Z=U^h_{\X,
T}\times_{H^h_{\X, T}} S$. In particular, for any $T$-equivariant closed immersion of $\X$ in a finite-dimensional $T$-module $V$, we have the following cartesian diagram:
$$
\xymatrix{
U^h_{\X,T} \ar[d] \ar@{^{(}->}[r] & U^h_{V,T} \ar[d] \\
H^h_{\X,T}  \ar@{^{(}->}[r] & H^h_{V,T}.
}
$$

\begin{lemma} \label{l0} Assume that $h(0)=1$.  Then we have the following commutative diagram:

$$
\xymatrix{
U^h_{\X,T} \ar[d]^{p_H} \ar[r]^{p_\X} & \X \ar[d]^{\pi} \\
H^h_{\X,T}  \ar[r]^{p_\X/\!/T} & \X/\!/T,
}
$$
where $p_\X$ is the projection, and $p_\X/\!/T$ assigns to any family its quotient by $T$.
The morphisms $p_\X$ and $p_\X/\!/T$ are projective.
\end{lemma}

\begin{proof} The commutativity of the diagram can  be  seen easily
by considering the corresponding morphisms of functors of points.
The morphism $p_\X/\!/T$ is projective by \cite[Lemma 3.3]{C}.
The morphism $U^h_{\X,T}\to H^h_{\X, T}\times_{\X/\!/T}\X$ is a
closed embedding, so the morphism $p_\X$ is projective. See also \cite[Proposition~3.15]{B}.

\end{proof}

\subsubsection{Definition of main component}\label{mc}

Let $$\X^s:=\bigcap_{\chi\in \Sigma}\bigcup_{f\in k[\X]_\chi}
\X_f.$$ Note that, if $x\in \X^s$, then $\overline{Tx}$ has the
following Hilbert function:

$$h_\Sigma(\chi):= \left\{ \begin{array}{rl}
 1 & \mbox{if}\ \ \chi\in \Sigma, \\
 0 & \mbox{otherwise.}
\end{array}\right. $$
Let $X_\Sigma$ denote the affine toric $T$-variety with the weight monoid $\Sigma$; then $\overline{Tx}\simeq X_{\Sigma}$ for $x\in \X^s$.
We denote by $H$ the invariant Hilbert scheme $H^{h_\Sigma}_{\X,T}$; it is called the {\it toric Hilbert scheme}  \cite{PS}. So for any $x\in \X^s$ we have the $k$-rational point $\overline{Tx}\in H$. Also denote $U:=U^{h_\Sigma}_{\X,T}$, $\H=\H^{h_\Sigma}_{\X,T}$.

Let $\X^s\times^T X_{\Sigma}$ be the quotient of the variety $\X^s\times X_{\Sigma}$ by the action of $T$ given by $t\cdot (x, y)=(t^{-1}x, ty)$.

\begin{lemma} \label{l1} There exists the geometric quotient $\pi^{s}:\X^{s}\to \X^{s}/T$ and an open embedding of $\X^{s}/T$ in
the toric Hilbert scheme $H$. Moreover, we have the following Cartesian diagram:

$$
\xymatrix{
\X^s\times^T X_{\Sigma} \ar[d] \ar@{^{(}->}[r] & U \ar[d]^{p_H}\\
\X^{s}/T  \ar@{^{(}->}[r] & H.
}
$$

\end{lemma}

\begin{proof}

Step 1. There exists the categorical quotient $\pi^{s}:\X^{s}\to
\X^{s}/T$. Moreover, $\pi^{s}$ is a locally trivial bundle.
Indeed, let $\chi_1,\ldots,\chi_r\in \Sigma$ be such that the
algebra $k[\X]$ is generated by $k[\X]_{\chi_i}$, $i=1\ldots, r$.
Then $\X^s$ is covered by  $T$-invariant open affine subschemes
$$\X_{f_1\cdot\ldots\cdot f_r}=\Sp(k[\X]_{f_1\cdot\ldots\cdot f_r}),$$ where $f_i\in k[\X]_{\chi_i}$,
$f_i\ne 0.$ Fix a basis $e_1,\ldots, e_d\in \Chi(T)$. Since the
characters $\chi_1,\ldots,\chi_r$ generate $\Sigma$ and,
consequently, $\Chi(T)$, we can choose non-zero elements $h_l\in
k[\X_{f_1\cdot\ldots\cdot f_r}]_{e_l}$. Then we have
$$k[\X_{f_1\cdot\ldots\cdot f_r}]\simeq k[\X_{f_1\cdot\ldots\cdot
f_r}]^T\otimes k[h_1^{\pm 1}, \ldots, h_d^{\pm 1}]\simeq k[\X_{f_1\cdot\ldots\cdot
f_r}]^T\otimes k[T].$$ Note that these isomorphisms  satisfy the
compatibility conditions and the variety $\X^{s}/T$ obtained by
gluing of affine charts $\X_{f}/T$ is separated. Indeed, we have
to show that the diagonal morphism $$\X_{fg}/T\to \X_{f}/T\times
\X_{g}/T$$ is closed for any $\X_{f}/T=\Sp (k[\X_{f}]^T)$ and
$\X_{g}/T=\Sp (k[\X_{g}]^T)$, where $f=f_1\cdot\ldots\cdot f_r$,
$g=g_1\cdot\ldots\cdot g_r$ and $f_i, g_i\in k[\X]_{\chi_i}$. This
is equivalent to showing that the corresponding homomorphism of
algebras $$k[\X_{f}]^T\otimes k[\X_{g}]^T\to k[\X_{fg}]^T$$ is
surjective. But this is clear since for any element
$\frac{h}{(fg)^k}\in k[\X_{fg}]^T$ we have
$\frac{h}{(fg)^k}=\frac{h}{f^{2k}}\frac{f^k}{g^k}$.

Step 2. Let us prove that $\X^{s}/T$ represents an open subfunctor of
$\H$. Consider the family
$$p^s:\X^s\times^TX_\Sigma\to \X^s/T.$$ We shall show that this is
the universal family over $\X^s/T$. Indeed, $\X^s\times^TX_\Sigma$
is a locally trivial bundle over $\X^s/T$ with fiber $X_\Sigma$.
Further, there is a canonical closed embedding $\X^s\times^T
X_\Sigma\subset \X\times (\X^s/T)$, which is locally given by the
surjective homomorphisms of algebras $$k[\X]\otimes
k[\X_{f_1\cdot\ldots\cdot f_r}]^T\to \bigoplus_{\chi\in \Sigma}
k[\X_{f_1\cdot\ldots\cdot f_r}]_\chi$$ ($\X^s\times^TX_\Sigma$ is
covered by the open affine subschemes $\Sp \bigoplus_{\chi\in
\Sigma} k[\X_{f_1\cdot\ldots\cdot f_r}]_\chi$).  This gives us an
element in $\underline{H}(\X^{s}/T)$, i.e., we have the
following cartesian diagram:

$$
\xymatrix{
\X^s\times^T X_{\Sigma} \ar[d] \ar[r] & U \ar[d]^{p_H}\\
\X^{s}/T  \ar[r] & H.
}
$$

Step 3.  It is clear that the image of $\X^{s}/T$  lies in the locus $H^s$ of $H$ where the fibers of the universal family are irreducible and reduced. By \cite[Theorem~12.1.1]{EGA}, it follows that $p_H^{-1}(H^s)$ is an open subscheme of $U$. Since $p_H^{-1}(H^s)$ is $T$-invariant and  $p_H$ maps closed $T$-invariant subsets to closed subsets, it follows that  $H^s$ is an open subscheme of $H$.
 We shall show that the morphism $\Phi:\X^{s}/T\to H^s$ is an isomorphism.
Given a morphism $\phi: S\to \X^{s}/T$, we have the following commutative diagram:

$$
\xymatrix{
Z \ar[d] \ar[r] & \X^s\times^TX_\Sigma \ar[d] \ar[r] & U^s \ar[d] \\
S  \ar[r] & \X^{s}/T \ar[r] & H^s,
}
$$
where $U^s:=H^s\times_{H}U$ and $Z:= S\times_{H^s}U^s$ is the image of $\phi$ in $\underline{H^s}(S)$. Note that all the squares of this  diagram are cartesian.
In particular, it follows that $Z=Z^s\times^TX_\Sigma$, where $Z^s:= S\times_{\X^{s}/T} \X^s = Z\cap(S\times \X^s)$.

Conversely, let us construct the inverse morphism $\Phi':H^s\to \X^{s}/T$. Let $Z\in \underline{H^s}(S)$. Consider $Z^s: = Z\cap(S\times \X^s)$;
then $Z^s/T=S$.  So the projection $Z^s\to \X^s$ defines a morphism of quotients $\{S\to \X^s/T\}\in \underline{\X^s/T}(S)$. Moreover, $Z^s=S\times_{\X^s/T} \X^s$ and $Z=Z^s\times^TX_\Sigma$. Thus we see that $\Phi\circ\Phi'={\rm Id}_{H^s}$ and $\Phi'\circ\Phi={\rm Id}_{\X^{s}/T}$.

\end{proof}

\begin{definition}
 The {\it main component} $H_0$ of the toric Hilbert scheme $H$ is the closure in $H$ of the image of $\X^{s}/T$.
\end{definition}

By Lemma \ref{l1}, it follows that the main component $H_0$ is an irreducible component of~$H$.  Consider also the {\it main component} of the universal family $U_0:=p_H^{-1}(H_0)$. Denote by $p_0$ the restriction of $p_H$ on $U_0$.

\begin{lemma}
 The main component $U_0$ is an irreducible component of $U$.
\end{lemma}

\begin{proof} It is sufficient to show that  $U_0$ is irreducible.
By \cite[Lemma 3.5]{C},  the dimension of any irreducible component of the
fibre $p^{-1}(x)$ equals $\dim T$ for any $x\in H$. Since $p_0$ is flat, this implies that the dimension
of any irreducible component $Z$ of $U_0$ is equal to $\dim \X$ \cite[Corollary 9.6]{Har}.
It follows that  $Z$  dominates $H_0$. Consequently, the intersection $Z\cap (\X^s\times^TX_\Sigma)$
is non-empty and $Z=\overline{Z\cap (\X^s\times^TX_\Sigma)}$.
This implies that $U_0=\overline{\X^s\times^TX_\Sigma}$ is irreducible.
\end{proof}

\subsection{The toric Chow morphism}\label{sec3-3}

First of all, let us recall the construction of the toric Chow morphism $H\to \X/_CT$.
This construction was given in \cite[Section 5]{HS} for the case when $\X$ is a finite-dimensional $T$-module, but it is almost the same in the case of an affine $T$-variety.

For any character $\chi\in \Sigma$, there exists $n>0$ such
that the $\Z$-graded algebra $R_{n\chi}:=k[\X]^{(n\chi)}$ is generated by its
elements of degree one.
 Note that the scheme $X/_{\sigma_{\chi}}T=\Proj R_{n\chi}$ represents the functor $\H^{h}_{\Sp{R_{n\chi}},\, k^{\times}}$, where
$$h(m):= \left\{ \begin{array}{rl}
 1 & \mbox{if}\ \ r\geq 0, \\
 0 & \mbox{otherwise}
\end{array}\right. $$(see \cite[Corollary~3.5]{C}).
The natural morphism of functors $\H\to \H^{h}_{\Sp{R_{n\chi}},\, k^{\times}}$ induces the canonical morphism of the schemes representing this functors $H\to \X/_{\sigma_{\chi}}T$ and it does not depend on the choice of $n$. Moreover, by \cite[Lemma 7.5]{HS}, these morphisms commute with the morphisms of the inverse system, so they define a canonical morphism $\Psi: H\to \X/_CT$.

\begin{lemma} There exists an open embedding $\X^{s}/T\subset (\X/_CT)_0.$ Moreover, we have the following cartesian diagram:
$$
\xymatrix{
\X^s\times^T X_\Sigma \ar[d] \ar@{^{(}->}[r] & W_X \ar[d]^{p_C} \\
\X^{s}/T  \ar@{^{(}->}[r] & (\X/_CT)_0.
}
$$
\end{lemma}
\begin{proof}
Note that in the inverse limit of GIT-quotients $\X/_{\!\!\lambda} T$, $\lambda\in \cQ$, it is sufficient to take the limit over $\lambda$ intersecting $\text{relint}(\omega)$. Moreover, for any such $\lambda$ the quotient $\X^s/T$ is an open subscheme in $X^{ss}_\lambda/\!/T$ and $X^s=q_\lambda^{-1}(\X^s/T)\subset X^{ss}_\lambda$.
Indeed, $\X^s$ is covered by the open affine $T$-stable subschemes $\X_{f_1\cdot\ldots\cdot f_r}$  (with the notation of the proof of Lemma~\ref{l1}). There exist $c_i>0$ such that $\sum c_i\chi_i=\chi$ for some $n>0$ and $\chi\in\ri{\lambda}$. Then $f=f_1^{c_1}\cdot\ldots\cdot f_r^{c_r}\in k[\X]_{\chi}$, and $\X_{f_1\cdot\ldots\cdot f_r}=\X_f=q_\lambda^{-1}(\X_f/T)\subset X^{ss}_\lambda$.

This implies that we have a Cartesian diagram
$$
\xymatrix{
\X^s \ar[d] \ar@{^{(}->}[r] & \X^{ss} \ar[d] \\
\X^{s}/T  \ar@{^{(}->}[r] & \X/_CT.
}
$$
This means that $\X^s=\X^{ss}\cap (\X\times \X^s/T)\subset \X\times \X/_CT.$
Note also that $\X^s\times^T X_\Sigma=\overline{\X^s}\subset \X\times \X^s/T$.  Finally, this implies that $\X^s\times^T X_\Sigma=W_\X \cap (\X\times \X^s/T)\subset \X\times \X/_CT$.

\end{proof}

In the following theorem we show that the toric Chow morphism $\Psi$ restricted on the main component lifts to a birational projective morphism $U^0\to W_X$.

\begin{theorem}\label{remw}
We have the following commutative diagram:

$$
\xymatrix{
         U_0 \ar[rr]^{\Phi}\ar[rd]_{p^0_X}\ar[ddd]_{p_H}  & & W_\X \ar[ddd]^{p_C}\ar[ld]^{p^W_{\X}} \\
          & \X \ar[ddd]_(.40){\pi} \\
          \\
          H_0 \ar[rr]_(.60){\Psi^0}\ar[rd]_{p^0_\X/\!/T} & & (\X/_C T)_0 \ar[dl]^{p^W_{\X}/\!/T} \\
          & \X/\!/T,
}
$$
where  $\Psi^0$ is the restriction  of $\Psi$ on
$H_0$ and $\Phi$ is the restriction of the morphism
$\Psi^0\times {\rm Id}_\X$ on $U_0\subset  H_0\times \X$. The morphisms $\Psi^0$, $\Phi$, $p^W_{\X}$ and $p^W_{\X}/\!/T$
are projective ant the first three of them are, moreover, birational.
\end{theorem}

\begin{proof}
The restriction of $\Psi^0$ on $\X^s/T\subset H_0$
is the identity map on $\X^s/T\subset (\X/_CT)_0$, thus
$\Psi^0$ is a birational morphism. Also, the restriction of
$\Psi^0\times {\rm Id}_\X$ on $\X^s\subset U_0$ is the
identity map on $\X^s\subset W_{\X,T}$, so $\Psi^0\times
{\rm Id}_\X$ maps $U_0$ birationally  on $W_\X$. By
construction of the morphisms $\Psi$ and $\Phi$, it
follows that $p^0_\X=p^W_{\X}\circ \Phi$ and $p^0_\X/\!/T
= p^W_{\X}/\!/T \circ \Psi^0$. Further,  by Lemma \ref{l0},
the morphisms $p^0_\X$ and $p^0_\X/\!/T$ are projective. It
follows that $\Phi$ and $\Psi^0$ are projective. The other assertions are obvious.
\end{proof}

\begin{corollary}\label{cor}
There exists a canonical birational projective morphism $U_0^{\norm}\to \widetilde{X}$.
\end{corollary}

\section{Subtorus actions on affine toric varieties}\label{sec4}

In this section $X=X_{\sigma}$ is a normal affine toric variety under the torus $\T$, $\sigma\subset \Lambda(\T)_{\Q}$ is a corresponding cone. We consider the action on $X$ of the subtorus $T\subset \T$. The natural homomorphism $\T\rightarrow \T/T$ induces the surjective map of vector spaces that are generated by the corresponding lattices of one-parameter subgroups $\alpha: \Lambda(\T)_{\Q}\rightarrow \Lambda(\T/T)_{\Q}$. The embedding of torus $T\hookrightarrow \T$ induces the surjective map of vector spaces $\beta: \mathfrak{X}(\T)_{\Q}\rightarrow \mathfrak{X}(T)_{\Q}$ such that $\beta(\sigma^{\vee})=\omega$, where $\omega\subset \mathfrak{X}(T)_{\Q}$ is the weight cone.

\subsection{The fan representation of the Altmann-Hausen family}

Let $\psi: \widetilde{X}\rightarrow Y$ be the Altmann-Hausen family of the $T$-action on $X$, and let $\varphi: \widetilde{X}\rightarrow X$ be the canonical morphism. We use the notations of Section \ref{cah} related to the construction of $\psi$.

Let $\T x$ be the open orbit of the $\T$-action on $X$. It can be easily seen that $X^{ss}\supset \T x$. Denote by $x_{\lambda}$, $x_C$ and $y$ the images of $x$ in $X/_{\lambda}T$, $(X/_C T)_0$ and $Y$ respectively. The quotient maps $q_{\lambda}$ induce actions of the torus $\T$ on the GIT-quotients $X/_{\lambda}T$. So $X/_{\lambda}T=\overline{\T x_{\lambda}}$ is the toric variety under the torus $\T/\T_{x_{\lambda}}$, where $\T_{x_{\lambda}}$ is the stabilizer. If $\lambda$ lies in the interior of $\omega$, general fibers of the morphism $q_{\lambda}$ contain a unique dense $T$-orbit, and consequently $\T_{x_{\lambda}}=T$. Note that the morphisms $p_{\lambda_1\lambda_2}$ in the GIT-system are $\T$-equivariant, so $Y=\overline{\T y}$ is a normal toric variety under the torus $\T/T$ (see \cite[Proposition~3.8]{CM}), and the morphisms $p_{\lambda}$ are $\T$-equivariant. Also $W_X=\overline{\T (x, x_C)}\subset X \times (X/_C T)_0$ is a toric (not necessarily normal) variety with the torus $\T$.

\begin{lemma}\label{ccr}
The fan $\mathcal{C}_{W_X}$ of the toric variety $W_X=\overline{\T (x, x_C)}\subset X\times (X/_C T)_0$ is the coarsest common refinement of the cone $\sigma$ and the quasifan $\alpha^{-1}(\mathcal{C}_{(X/_C T)_0})$.
\end{lemma}

\begin{proof}
  Note that we can consider a fan of the toric variety $W_X$ (which is not necessarily normal), because there exists a $\T$-invariant affine covering $W_X=X \times U_i$, where $(X/_C T)_0=\bigcup U_i$ is an affine $\T$-invariant covering of the variety $(X/_C T)_0$. The existence of the last one immediately follows from the existence of a $\T$-equivariant closed embedding of variety $(X/_C T)_0$ in the direct product of GIT-factors.

  Lemma follows now from Proposition \ref{per} and the definition of the variety $W_X$.
\end{proof}

Let $\cQ\subset \mathfrak{X}(T)_{\Q}$ be the GIT-fan. Denote by~$\cQ^0$ the set of $\lambda\in\cQ$ satisfying the condition $\lambda\cap \text{relint}(\omega)\ne \emptyset$.
We have already seen that for $\lambda\in\cQ^0$ the corresponding GIT-quotients $X/_{\lambda}T$ are toric varieties under the torus $\T/T$. The variety $Y$ is also toric under the torus $\T/T$. The construction of the fans of these varieties is described in \cite{CM}. In the framework of that paper a more general case of a semiprojective toric variety $X$ corresponding to the normal fan of some polyhedra $P\subset \mathfrak{X}(\T)_{\Q}$. The case of an affine toric variety $X=X_{\sigma}$ considered in \cite[Section~5]{C} (in this case $P=\sigma^{\vee}$). We shall continue our exposition following \cite{C}.

For all $\chi\in \omega$, denote $P_{\chi}=\beta^{-1}(\chi)\cap\sigma^{\vee}$. Let $\lambda\in \cQ^0$. Then, for all $\chi\in\text{relint}(\lambda)$, the polyhedra $P_{\chi}$ have the same normal fan $\mathcal{N}_{\lambda}\subset \Lambda(\T/T)$ (see \cite[Remark~4.12]{C}) which coincides with the fan of the toric variety $X/_{\lambda}T$. The fan of the toric variety $Y$ is the coarsest common refinement of all the fans $\mathcal{N}_{\lambda}$, where $\lambda\in\cQ^0$, so it is the normal fan of the Minkowski sum $$\sum_{\lambda\in\Lambda^0} P_{\chi_{\lambda}},$$ where $\chi_{\lambda}\in\text{relint}(\lambda)$.

 Having the description of the fan $\mathcal{C}_Y=\mathcal{C}_{(X/_C T)_0}$, the following result allows us to describe the Altmann-Hausen family in the toric case.

\begin{theorem}\label{veer}
Let $X=X_{\sigma}$ be the toric variety with a big torus $\T$, and let $T\subset \T$ be a subtorus. Then the variety $\widetilde{X}$ constructed for the action of $T$ on $X$ is a normal toric variety with the torus $\T$. Its fan $\mathcal{C}_{\widetilde{X}}\subset \Lambda(\T)$ is the coarsest common refinement of the quasifan $\alpha^{-1}(\mathcal{C}_{(X/_C T)_0})$ and the cone $\sigma$, where $\alpha: \Lambda(\T)\rightarrow \Lambda(\T/T)$ is the natural map.
\end{theorem}
\begin{proof}By Proposition \ref{main}, the toric variety $\widetilde{X}$ is isomorphic to the normalization of the variety $W_X$, so their fans $\mathcal{C}_{\widetilde{X}}$ and $\mathcal{C}_{W_X}$ coincide. Now our statement follows directly from Lemma \ref{ccr}.
\end{proof}

The fact that the fans $\mathcal{C}_{\widetilde{X}}$ and $\mathcal{C}_{W_X}$ coincide (and so Theorem \ref{veer}) could be proved by methods of toric geometry without using Proposition \ref{main}. We shall give this proof below. Moreover, one can prove Proposition \ref{main} in the toric case using Theorem \ref{veer} and then pass on to general case.

\begin{proof}[Another proof of Theorem \ref{veer}]
The morphism $\varphi: \widetilde{X}\rightarrow X$ is proper, so, by Proposition \ref{prop}, the support of the fan $C_{\widetilde{X}}$ coincides with $\sigma$. The morphism $\psi: \widetilde{X}\rightarrow Y$ is toric, hence, by Proposition \ref{tormorph}, the image $\alpha_{\Q}(\tau)$ of any cone $\tau\in \mathcal{C}_{\widetilde{X}}$ is contained in some cone $\delta \in \mathcal{C}_Y$. Also the morphism $\psi$ is a good quotient, so, by \cite[Theorem~4.1]{S}, we have that $\delta=\alpha(\tau)$ and $\tau=\alpha^{-1}(\delta)\cap \sigma$, which exactly means that the fan $C_{\widetilde{X}}$ is the coarsest common refinement of the quasifan $\alpha^{-1}(\mathcal{C}_{(X/_C T)_0})$ and the cone $\sigma$. By Lemma \ref{ccr}, it follows that the fans $\mathcal{C}_{W_X}$ and $C_{\widetilde{X}}$ coincide.
\end{proof}

\begin{example}\label{ahex}
Let us consider the action of the one-dimensional torus $T$ on the four-dimensional affine space $X=\A^4$ given by $t\cdot (x_1,x_2,x_3,x_4)=(tx_1,tx_2,t^{-1}x_3,t^{-1}x_4)$. Our main purpose is to construct the fan $\C_{\widetilde{X}}$ of the toric variety $\widetilde{X}$. Along the way, we will construct the fans of the toric varieties $X/_{\lambda}T$ and $Y$, where $\lambda \in \cQ^0$.

The GIT-fan $\cQ$ consists of the three cones $\lambda_{-1}$, $\lambda_0$, and $\lambda_1$:
\begin{center}
\begin{picture}(90,20)
\linethickness{0.34mm}
\put(0,10){\line(1,0){90}}
\put(45,10){\circle*{3}}
\put(42,0){$0$}
\put(63,0){$1$}
\put(17,0){$-1$}
\put(15,13){$\lambda_{-1}$}
\put(42,13){$\lambda_0$}
\put(68,13){$\lambda_1$}
\end{picture}
\end{center}

We will consider the torus $T$ as the subtorus of the four-dimensional torus $\T$ that acts on $X$ by the rescaling of coordinates. Then $X=X_{\sigma}$, where $\sigma\subset\Lambda (\T)_{\Q}$ is the cone generated by the standard basis $e_1$, $e_2$, $e_3$, $e_4$ of $\Lambda (\T)_{\Q}$.

There is a natural immersion $\Lambda(T)_{\Q}\hookrightarrow \Lambda(\T)_{\Q}$. We will identify the space $\Lambda(\T/T)_{\Q}$ with the orthogonal complement of $\Lambda(T)_{\Q}$ in the space $\Lambda(\T)_{\Q}$. Then the natural map $\alpha: \Lambda(\T)_{\Q}\rightarrow \Lambda(\T/T)_{\Q}$ will be the orthogonal projection.

In the coordinates of the standard basis of $\Lambda(\T)_{\Q}$ denote $\rho_1=(3,-1,1,1)$, $\rho_2=(1,1,-1,3)$, $\rho_3=(-1,3,1,1)$, $\rho_4=(1,1,3,-1)$ and $\rho=(1,1,1,1)$. The fans of the toric varieties $X/_{\lambda_0}T$, $X/_{\lambda_{-1}}T$ and $X/_{\lambda_1}T$ are the normal fans of the polyhedra $P_0$, $P_{-1}$ and $P_1$ respectively. After some calculations one can verify that the fan $\C_{X/_{\lambda_0}T}=\mathcal{N}_{\lambda_0}$ consists of one maximal cone $$\langle\rho_1,\rho_2,\rho_3,\rho_4\rangle_{\Q_{\geq 0}},$$ the fan $\mathcal{C}_{X/_{\lambda_{-1}}T}=\mathcal{N}_{\lambda_{-1}}$ consists of two maximal cones
$$\langle \rho_1,\rho_2,\rho_3\rangle_{\Q_{\geq 0}}\text{ and }\langle \rho_3,\rho_4,\rho_1\rangle_{\Q_{\geq 0}}$$ and the fan $\mathcal{C}_{X/_{\lambda_1}T}=\mathcal{N}_{\lambda_1}$ consists of two maximal cones
$$\langle \rho_4,\rho_1,\rho_2\rangle_{\Q_{\geq 0}}\text{ and }\langle \rho_2,\rho_3,\rho_4\rangle_{\Q_{\geq 0}}.$$ The fan $\mathcal{C}_Y$, as the coarsest common refinement of $\C_{X/_{\lambda_0}T}$, $\C_{X/_{\lambda_{-1}}T}$ and $\C_{X/_{\lambda_1}T}$, consists of four maximal cones
$$\langle \rho_1,\rho_2,\rho\rangle_{\Q_{\geq 0}},\, \langle\rho_2,\rho_3,\rho\rangle_{\Q_{\geq 0}},\, \langle\rho_3,\rho_4,\rho\rangle_{\Q_{\geq 0}},\,\text{and }\langle\rho_4,\rho_1,\rho\rangle_{\Q_{\geq 0}}.$$

\medskip

\begin{center}
\begin{picture}(200,160)
\put(62,160){\textbf{GIT-system}}
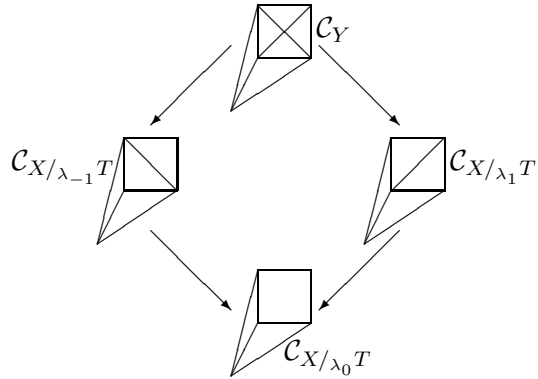
\put(90,20){\line(0,1){20}}
\put(90,20){\line(1,0){20}}
\put(110,40){\line(0,-1){20}}
\put(110,40){\line(-1,0){20}}
\put(80,0){\line(1,2){10}}
\put(80,0){\line(1,4){10}}
\put(80,0){\line(3,2){30}}
\small{\put(100,6){$\mathcal{C}_{X/_{\lambda_0}T}$}
\put(-2,78){$\mathcal{C}_{X/_{\lambda_{-1}}T}$}
\put(162,78){$\mathcal{C}_{X/_{\lambda_1}T}$}
\put(112,128){$\mathcal{C}_Y$}
}

\put(40,70){\line(0,1){20}}
\put(40,70){\line(1,0){20}}
\put(40,90){\line(1,-1){20}}
\put(60,90){\line(0,-1){20}}
\put(60,90){\line(-1,0){20}}
\put(30,50){\line(1,2){10}}
\put(30,50){\line(1,4){10}}
\put(30,50){\line(3,2){30}}

\put(140,70){\line(0,1){20}}
\put(140,70){\line(1,0){20}}
\put(140,70){\line(1,1){20}}
\put(160,90){\line(0,-1){20}}
\put(160,90){\line(-1,0){20}}
\put(130,50){\line(1,2){10}}
\put(130,50){\line(1,4){10}}
\put(130,50){\line(3,2){30}}

\put(90,120){\line(0,1){20}}
\put(90,120){\line(1,0){20}}
\put(90,120){\line(1,1){20}}
\put(90,140){\line(1,-1){20}}
\put(110,140){\line(0,-1){20}}
\put(110,140){\line(-1,0){20}}

\put(80,100){\line(1,2){10}}
\put(80,100){\line(1,4){10}}
\put(80,100){\line(3,2){30}}

\put(80,125){\vector(-1,-1){30}}
\put(113,125){\vector(1,-1){30}}
\put(50,55){\vector(1,-1){30}}
\put(143,55){\vector(-1,-1){30}}

\end{picture}
\end{center}
The fan $\mathcal{C}_{\widetilde{X}}$ is the coarsest common refinement of the quasifan $\alpha^{-1}(\mathcal{C}_{Y})$ and the cone~$\sigma$. Denote $\mu_1=(1,1,0,0)$, $\mu_2=(0,0,1,1)$. Calculations show that $\mathcal{C}_{\widetilde{X}}$ consists of four maximal cones
$$\langle e_1, e_3, \mu_1,\mu_2\rangle_{\Q_{\geq 0}},\,\langle e_1,e_4,\mu_1,\mu_2\rangle_{\Q_{\geq 0}},\,\langle e_2, e_3, \mu_1,\mu_2\rangle_{\Q_{\geq 0}}\,\,\text{and }\langle e_2,e_4,\mu_1,\mu_2\rangle_{\Q_{\geq 0}}.$$ On the picture below you can see polyhedral slice of this fan by the hyperplane $x_1+x_2+x_3+x_4=1$.

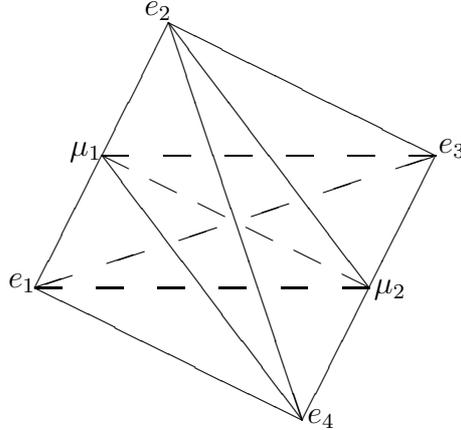
\begin{center}
\begin{picture}(210,155)
\put(71,153){$e_2$}
\put(20,50){$e_1$}
\put(43,100){$\mu_1$}
\put(181,101){$e_3$}
\put(132,-1){$e_4$}
\put(157,48){$\mu_2$}
\put(30,50){\line(1,2){50}}
\put(80,150){\line(2,-1){100}}
\put(80,150){\line(1,-3){50}}
\put(55,100){\line(3,-4){75}}
\put(80,150){\line(3,-4){75}}
\put(180,100){\line(-1,-2){50}}
\put(130,0){\line(-2,1){100}}
\multiput(55,100)(18,-9){6}%
{\line(2,-1){10}}
\multiput(55,100)(23,0){6}%
{\line(1,0){10}}
\multiput(30,50)(23,0){6}%
{\line(1,0){10}}
\multiput(30,50)(27,9){6}%
{\line(3,1){14}}

\end{picture}
\end{center}

\end{example}

\subsection{The fan of the main component of the toric Hilbert scheme}
 Finally, we complete a description of all the fans of the toric varieties from the commutative diagram (\ref{cd}) given in Section~\ref{intro} by recalling some results from \cite{C}.

Consider the main component $H_0$ of the toric Hilbert scheme and the main component $U_0$ of the universal family in the toric situation with their structure of reduced schemes. As shown in \cite[Proposition~3.6]{C}, these varieties are toric under the tori $\T/T$ and $T$ respectively. Our purpose is to recall the construction of the fans corresponding to toric varieties $H_0$ and $U_0$.

For all $\chi\in \omega\cap \Chi(T)$ denote $P^I_{\chi}:=\conv (\beta^{-1}(\chi)\cap \mathfrak{X}(\T))$.
\begin{remark}\label{int}
The equality $P^I_{\chi}=P_{\chi}$ holds if and only if the polyhedra $P_{\chi}$ have integer vertices. Such characters $\chi$ are called \emph{integer}.
\end{remark}
We call two polyhedra \emph{equivalent} if they have the same normal fan.
\begin{proposition}\cite[Proposition~4.9(2)]{C} There are only finitely many non-equivalent polyhedra $P^I_{\chi}$ for $\chi\in \omega$. The fan $\C_{H_0}$ corresponding to the variety  $H_0$ is the normal fan of the Minkowski sum over a system of representatives of the equivalent classes of these polyhedra.
\end{proposition}
 Of course, it is again equivalent to say that $\C_{H_0}$ is the coarsest common refinement of the normal fans of all $P_{\chi}$.

The variety $U_0$ coincides with $\overline{\T (y, x)}\subset H_0\times X$, where $x\in X$ lies in the open $\T$-orbit and $y:=\overline{T x}\in H_0$.
From Proposition \ref{per} it follows that the fan $\C_{U_0}$ is the coarsest common refinement of the fan $\C_{H_0}$ and $\sigma$.

Now we know the fans of all varieties placed in the vertices of the square from the commutative diagram (\ref{cd}). By Remark \ref{int} it follows that the fan $\C_{H_0}$ refines the fan $\C_{(X/_CT)_0}$ and the fan $\C_{U_0}$ refines the fan $\C_{W_X}$, what illustrates that morphisms $\Psi^0$ and $\Phi$ are toric.
\begin{remark}
From construction of all these fans or straightforwardly from the properties of morphisms $\Psi^0$, $\Phi$ and $p^W_{\X}$ it follows that $|\C_{H_0}|=|\C_{(X/_CT)_0}|$ and $|\C_{U_0}|=|\C_{W_X}|=\sigma$.
\end{remark}

In Example \ref{ahex} all characters $\chi\in\omega\cap\Chi(T)$ are integer, since the fans of the toric varieties $\C_{H_0}$ and $\C_{(X/_CT)_0}$ coincide (and so, certainly, the fans $\C_{U_0}$ and $\C_{W_X}$ also coincide). But this is not true in general. The example below was considered in \cite{C}. There were constructed the fans $\C_{H_0}$ and $\C_{(X/_CT)_0}$, and it was shown that they do not coincide in the considered situation. We shall complete this example constructing the fans $\C_{U_0}$ and $\C_{W_X}=\C_{\widetilde{X}}$.

\begin{example}\cite[Example~5.3]{C}
Let $X_{\sigma}=\A^3$, $\T=(\G)^3$ act by rescaling of coordinates, and $T=\G$ act by $t\cdot (x_1,x_2,x_3)=(tx_1,tx_2,t^2 x_3)$. Denote by $e_1=(1,0,0)$, $e_2=(0,1,0)$, $e_3=(0,0,1)$ the standard basis of $\Lambda(\T)_{\Q}$, and let $\nu_1,\,\nu_2,\,\nu_3$ be its dual basis in $\Chi(\T)_{\Q}$.

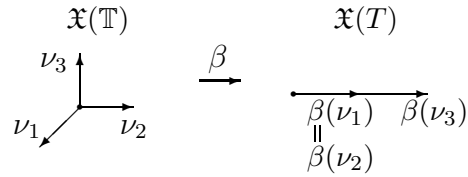
\medskip

\begin{center}
\begin{picture}(150, 55)
{\small \put(20, 50){$\Chi(\T)$}}

\put(25, 20){\circle*{2}}

\put(25, 20){\vector(-1, -1){15}} \put(0, 10){$\nu_1$}

\put(25, 20){\vector(1, 0){20}} \put(40, 10){$\nu_2$}

\put(10, 35){$\nu_3$} \put(25, 20){\vector(0, 1){20}}

\put(70, 30){\vector(1, 0){15}} \put(73, 35){$\beta$}

{\small \put(120, 50){$\Chi(T)$}

\put(110, 15){$\beta(\nu_1)$} \put(110, -2){$\beta(\nu_2)$} \put(113, 7){$\line(0,1){6}$} \put(115,
7){$\line(0,1){6}$}

\put(145, 15){$\beta(\nu_3)$}}

\put(105, 25){\vector(1, 0){25}}

\put(105, 25){\vector(1, 0){50}}

\put(105, 25){\circle*{2}}

\end{picture}
\end{center}

It is easy to see that the fan $\C_{H_0}$ is the normal fan of the trapezoid $P_3$, whereas the fan $\C_{(X/_CT)_0}$ is the normal fan of the triangle $P_2$.
As in Example~\ref{ahex},
we will identify the space $\Lambda(\T/T)_{\Q}$ with the orthogonal complement of $\Lambda(T)_{\Q}$ in the space $\Lambda(\T)_{\Q}$.

Denote $\rho_1=(1,1,-1)$, $\rho_2=(-5,1,2)$, $\rho_3=(1,-5,2)$ in the coordinates of the standard basis of $\Lambda (\T)_{\Q}$. All these vectors are lying in $\Lambda(\T/T)_{\Q}$, and it is easy to verify that $\C_{(X/_CT)_0}$ consists of three maximal cones $$\langle\rho_1, \rho_2\rangle_{\Q_{\geq 0}},\,\langle\rho_2, \rho_3\rangle_{\Q_{\geq 0}},\,\text{and}\ \langle\rho_1, \rho_3\rangle_{\Q_{\geq 0}}$$
and $\C_{H_0}$ consists of four maximal cones
$$\langle\rho_1, \rho_2\rangle_{\Q_{\geq 0}},\,\langle\rho_1, \rho_3\rangle_{\Q_{\geq 0}},\,\langle -\rho_1, \rho_2\rangle_{\Q_{\geq 0}},\,\text{and}\ \langle -\rho_1, \rho_3\rangle_{\Q_{\geq 0}}.$$

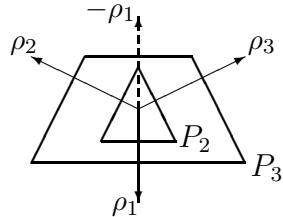
\begin{center}
\begin{picture}(100,80)
\put(50,40){\vector(0,-1){35}}
\put(50,40){\vector(-2,1){40}}
\put(50,40){\vector(2,1){40}}
\multiput(50,40)(0,5){6}%
{\line(0,1){3}}
\put(50,70){\vector(0,1){5}}
\thicklines
\put(10,20){\line(1,0){80}}
\put(30,60){\line(1,0){40}}
\put(50,56){\line(-1,-2){14}}
\put(50,56){\line(1,-2){14}}
\put(36,28){\line(1,0){28}}
\put(10,20){\line(1,2){20}}
\put(90,20){\line(-1,2){20}}
\put(92,15){$P_3$}
\put(65,26){$P_2$}
{\small \put(40,2){$\rho_1$}
\put(30,75){$-\rho_1$}
\put(2,63){$\rho_2$}
\put(91,63){$\rho_3$}
}
\end{picture}
\end{center}

We see that the fan $\C_{H_0}$ subdivides the fan $\C_{(X/_CT)_0}$, and, in particular, they do not coincide. If we intersect the corresponding quasifans in $\Lambda(\T)_{\Q}$ with the cone $\sigma=\langle e_1,e_2,e_3 \rangle_{\Q_{\geq 0}}$, we get the fans $\C_{U_0}$ and $\C_{W_X}$ respectively. Denote $\kappa=(1,1,2)$, $\mu_1=(1,1,0)$, $\mu_2=(1,0,2)$ and $\mu_3=(0,1,2)$. Then $\C_{W_X}$ consists of three maximal cones
$$\langle\kappa, \mu_1, e_1, \mu_2\rangle_{\Q_{\geq 0}},\,\langle\kappa, \mu_1, e_2, \mu_3 \rangle_{\Q_{\geq 0}},\,\text{and}\ \langle\kappa, \mu_2, e_3, \mu_3 \rangle_{\Q_{\geq 0}}$$
and $\C_{U_0}$ consists of four maximal cones
$$\langle\kappa, \mu_1, e_1, \mu_2\rangle_{\Q_{\geq 0}},\,\langle\kappa, \mu_1, e_2, \mu_3 \rangle_{\Q_{\geq 0}},\,\langle\kappa, e_3, \mu_2 \rangle_{\Q_{\geq 0}},\,\text{and}\ \langle\kappa, e_3, \mu_3 \rangle_{\Q_{\geq 0}}.$$

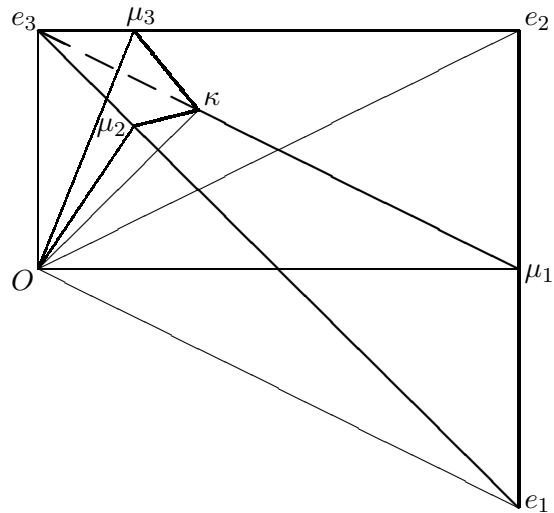
\begin{center}
\begin{picture}(200,200)

\put(10,100){\line(0,1){90}}
\put(10,100){\line(2,1){180}}
\put(10,100){\line(1,0){180}}
\put(10,100){\line(2,-1){180}}
\put(70,160){\line(-1,-1){60}}
\qbezier(10,100)(10,100)(46,190)
\qbezier(10,100)(10,100)(46,154)
{\small \put(0,192){$e_3$}
\put(192,192){$e_2$}
\put(192,10){$e_1$}
\put(192,97){$\mu_1$}
\put(72,162){$\kappa$}
\put(43,194){$\mu_3$}
\put(32,152){$\mu_2$}
\put(0,92){$O$}
}
\thicklines
\put(190,10){\line(0,1){180}}
\put(10,190){\line(1,0){180}}
\put(10,190){\line(1,-1){180}}
\put(70,160){\line(2,-1){120}}
\qbezier(70,160)(58,175)(46,190)
\qbezier(70,160)(70,160)(46,154)
\multiput(70,160)(-16,8){4}
{\line(-2,1){10}}
\put(10,190){\line(2,-1){10}}
\end{picture}
\end{center}

\end{example}

\section*{Acknowledgements}
The authors are grateful to Ivan V. Arzhantsev for his attention to this work and fruitful discussions. We also thank Dmitri A. Timashev for many important comments. The first author is
 thankful to Michel Brion for stimulating ideas and useful
 discussions. The second author thanks Roman Gayduk and Mikhail Gorsky for helpful comments. The authors are grateful to the referee for important comments which
leaded to improvement of the article.

\normalsize

\bigskip


\begin{thebibliography}{99}

\bibitem[AB05]{AB} \textsc{V. Alexeev, M. Brion},  Moduli of affine schemes with reductive group action, J.\ Algebraic \ Geom., 2005, 14, no. 1, 83--117

\bibitem[AH06]{AH}  \textsc{K. Altmann,  J. Hausen},  Polyhedral divisors and algebraic torus
actions, Math. Ann., 2006, 334, 557--607

\bibitem[AH07]{ArH} \textsc{I. Arzhantsev, J. Hausen}, On the multiplication map of a multigraded algebra, Math. Res. Lett., 2007, 14, no. 1, 129--136

\bibitem[BH06]{BH} \textsc{F. Berchtold, J. Hausen}, GIT-equivalence beyond the ample cone,  Michigan Math. J., 2006, 54, 483--515

\bibitem[Be08]{Be} \textsc{J. Bertin}, The punctual Hilbert scheme: an introduction, available at http://cel.archives-ouvertes.fr/cel-00437713/en/

\bibitem[Br11]{B} \textsc{M. Brion}, Invariant Hilbert schemes, preprint available at http://arxiv.org/abs/1102.0198

\bibitem[Ch08]{C} \textsc{O. Chuvashova},  The main component of the Hilbert scheme, Tohoku
Math. J. (2), 2008, 60, no. 3, 365--382

\bibitem[CLS11]{CLS} \textsc{D. Cox, J. Little, H. Schenck}, Toric Varieties, Grad. Stud. Math., 124, Amer. Math. Soc., 2011

\bibitem[CM07]{CM} \textsc{A. Craw, D. Maclagan}, Fiber fans and toric quotients, Discrete Comput. Geom., 2007, 37, no.~2, 251--266

\bibitem[EH00]{EH} \textsc{D. Eisenbud, J. Harris}, The geometry of schemes, Grad. Texts in Math., 197, Springer-Verlag, New York, 2000

\bibitem[Fu93]{Ful} \textsc{W. Fulton}, Introduction to toric varieties, Ann. of Math.
 Stud., 131, Princeton  Univ. Press, 1993

\bibitem[Gr67]{EGA} \textsc{A. Grothendieck}, \'El\'ements de g\'eom\'etrie alg\'ebrique. IV. \'Etude locale des sch\'emas et des morphismes de sch\'emas III,  Inst. Hautes \'Etudes Sci. Publ. Math., 32, 1967

\bibitem[HS04]{HS} \textsc{M. Haiman, B. Sturmfels}, Multigraded Hilbert schemes, J.\ Algebraic\ Geom., 2004, 13, 725--769

\bibitem[Ha77]{Har} \textsc{R. Hartshorne}, Algebraic geometry, Grad. Texts in Math., 52, Springer-Verlage, New York, 1977

\bibitem[KSZ91]{KSZ} \textsc{M. Kapranov, B. Sturmfels, A. Zelevinsky}, Quotients of toric varieties, Math. Ann., 1991, 290, 644--655

\bibitem[MFK94]{MF} \textsc{D. Mumford, J. Fogarty, F. Kirwan}, Geometric invariant theory, 3rd edition, Ergeb. Math. Grenzgeb., 34, Springer-Verlag, Berlin, 1994

\bibitem[Od88]{Oda} \textsc{T. Oda}, Convex Bodies and Algebraic Geometry (An Introduction to the Theory of
Toric Varieties), Ergeb. Math. Grenzgeb., 15, Springer-Verlag, Berlin, 1988

\bibitem[PS02]{PS} \textsc{I. Peeva, M. Stillman}, Toric Hilbert schemes, Duke Math.\ J., 2002, 111, 419--449

\bibitem[Sw99]{S} \textsc{J. Swiecicka}, Quotients of toric varieties by actions of subtori, Colloq.
Math., 1999, 82, no. 1,  105--116

\bibitem[Vo10]{V} \textsc{R. Vollmert}, Toroidal embeddings and polyhedral divisors, Int. J. Algebra, 2010, 4, no. 5–8, 383–-388

\bibitem[Zi95]{Z} \textsc{G. Ziegler}, Lectures on Polytopes, Grad. Texts in Math., 152, Springer-Verlag, New York, 1995

\end{thebibliography}
\end{document}